\documentclass{aptpub}

\usepackage{color,epsfig}
\usepackage{url}

\authornames{A. L. STOLYAR} 
\shorttitle{Diffusion scale steady-state tightness} 

\newcommand{\BE}{{\mathbb E}}
\newcommand{\BP}{{\mathbb P}}

\newcommand{\CSSS}{{\mathcal S}}
\newcommand{\CC}{{\mathcal C}}

\newcommand{\CJJJ}{{\mathcal J}}
\newcommand{\CI}{{\mathcal I}}
\newcommand{\CE}{{\mathcal E}}

\providecommand{\abs}[1]{\left\lvert#1\right\rvert}

\newcommand{\floor}[1]{\lfloor#1\rfloor}

\newcommand{\beql}[1]{\begin{equation}\label{#1}}
\newcommand{\eqn}[1]{(\ref{#1})}

\begin{document}

\title{Diffusion scale tightness\\ of invariant distributions\\ of a large-scale flexible service system}

\authorone[Bell Labs, Alcatel-Lucent]{A. L. Stolyar}

\addressone{Bell Labs, Alcatel-Lucent,
600 Mountain Ave., 2C-322,
Murray Hill, NJ 07974, USA \texttt{stolyar@research.bell-labs.com}}

\begin{abstract}
A large-scale service system with multiple customer classes and multiple server pools is considered, with the mean service time depending both on the customer class and server pool. The allowed activities (routing choices) form a tree (in the graph with vertices being both customer classes and server pools). We study the behavior of the system under a {\em Leaf Activity Priority} (LAP) policy, introduced in \cite{SY2012}.
An asymptotic regime is considered, where the arrival rate of customers and number of servers in each pool tend to infinity in proportion to a scaling parameter $r$, while the overall system load remains strictly subcritical. We prove tightness of diffusion-scaled 
(centered at the equilibrium point and scaled down by $r^{-1/2}$) invariant distributions.
As a consequence, we obtain a limit interchange result: the 
limit of diffusion-scaled invariant distributions is equal to the invariant distribution of the limiting diffusion process.
\end{abstract}

\vspace{0.05in}


\vspace{0.05in}



\section{Introduction}

Large-scale heterogeneous flexible service systems naturally arise as models of
large call/contact centers \cite{AksinArmonyMehrotra,GansKooleMandelbaum}, large computer farms (used in network cloud data centers), etc.
More specifically, in this paper we consider a
service system with multiple customer and server types (or classes), 
where the arrival rate of class $i$ customers is $\Lambda_i$, the service rate of a class $i$ customer by a type $j$ server is $\mu_{ij}$, 
and the server pool $j$ size (the number of type $j$ servers) is $B_j$.
It is important that the service rate $\mu_{ij}$
in general depends on both the customer type $i$ and server type $j$.
Customers waiting for service are queued, and they 
cannot leave the system before their service is complete.
The system is ``large-scale'' in the sense that the input rates $\Lambda_i$
and pool sizes $B_j$ are large. More precisely, we will consider the 
``many-servers'' asymptotic regime, 
in which the arrival rates $\Lambda_i$ and pool sizes $B_j$ scale up
to infinity in proportion to a scaling parameter $r$, i.e. $\Lambda_i=\lambda_i r$,
$B_j = \beta_j r$, while the service rates $\mu_{ij}$ remain constant.
Furthermore, in this paper we assume that (appropriately defined) system capacity
exceeds the (appropriately defined) traffic load by $O(r)$ amount -- i.e. the system
is {\em strictly subcritically} loaded. 
(This is different from {\em Halfin-Whitt} ``many-servers'' regime,
in which the capacity exceeds the load by $O(\sqrt{r})$.)

If under a given control policy the system is stable, i.e., roughly speaking,
it has a stationary distribution such that the queues are stochastically bounded, 
then the average number of occupied 
servers in a stationary regime is of course $O(r)$. 
A ``good'' control policy would keep the steady-state 
system state within $O(\sqrt{r})$ of 
its {\em equilibrium point}, which depends 
on the system parameters and on the policy itself.
More precisely, this means that the sequence (in $r$) of the
system stationary distributions, centered at equilibrium point
and scaled down by $r^{-1/2}$, is tight.
We will refer to this property as {\em $r^{1/2}$-scale, or diffusion-scale,
 tightness (of invariant distributions)}.

It is typically easy to construct a policy ensuring 
the diffusion-scale tightness,
{\em if the system 
parameters $\lambda_i$ and $\mu_{ij}$ are known in advance}.
(It is natural to assume that pool sizes {\em are} available 
to any control policy.)
In this case the equilibrium point
can be computed in advance,
and then the appropriate fractions of each input flow routed
to appropriate server pools. (See discussion in \cite{SY10}.)
It is much more challenging to establish this property 
for ``blind'' policies, which do not ``know'' 
parameters $\lambda_i$ and $\mu_{ij}$.
In fact, as shown in \cite{SY10}, 
under a very natural {\em Largest-Queue, Freest-Server Load Balancing} (LQFS-LB) algorithm
(which is a special case of the QIR policy in \cite{Gurvich_Whitt}),
the diffusion-scale tightness does not hold in general.
LQFS-LB assumes that the set of allowed "activities" $(ij)$ (those
with $\mu_{ij}>0$) is known (while the 
actual $\mu_{ij}$ values may not be)
and forms a tree in the graph with vertices being customer and server types 
-- let us refer to this
as the {\em tree assumption}; otherwise, the LQFS-LB is blind.

Another example of a blind policy (which also requires the tree assumption) is the 
{\em Leaf Activity Priority} (LAP) algorithm, introduced
in \cite{SY2012}.
(LAP policy is formally defined in Section~\ref{section:model}, and its
features and assumptions, including the tree assumption,
are discussed in Section~\ref{sec-lap-discussion}.)
It was shown in \cite{SY2012}, that LAP ensures $r^{1/2+\epsilon}$-scale
tightness of invariant distributions, for any $\epsilon>0$.

\subsection{Main result and contributions} 

In this paper we prove that, in fact, the diffusion-scale (i.e., $r^{1/2}$-scale) 
tightness of invariant distributions holds under the LAP algorithm.
We use the weaker, $r^{1/2+\epsilon}$-scale tightness result
of \cite{SY2012} as a starting point, and make an additional step 
to obtain the diffusion scale tightness from it.
This additional step is non-trivial and is not a simple extension 
of the technique in \cite{SY2012}. More specifically, for establishing 
$r^{1/2+\epsilon}$-scale tightness in \cite{SY2012}, it suffices to work
with the process under several {\em fluid} scalings ("standard" fluid
scaling for many-servers regime, as well as {\em hydrodynamic} and 
{\em local-fluid} scalings). In this paper, to prove the diffusion-scale
tightness, we need to also work with the process under diffusion scaling.
Informally speaking, the major technical challenge here is
in showing that the diffusion-scaled process
is uniformly close to the corresponding limiting diffusion process
{\em on time intervals of the length increasing with $r$, namely $O(\log r)$-long intervals.}

The diffusion-scale tightness under LAP  in turn
implies a limit interchange property: the limit of (diffusion-scaled)
invariant distributions is equal to the invariant distribution
of the limit (diffusion) process. Proving this limit interchange
in many-servers regime is very challenging,
especially for general
models with multiple customer and server classes; the reason is
 precisely 
the difficulty of establishing the diffusion-scale tightness.

Perhaps more important than establishing the tightness and limit interchange
specifically for the LAP policy, is the fact that our technique seems quite
generic, and may apply to other policies and/or other many-servers
models. Speaking very informally,
the combination of results and proofs in \cite{SY2012} and this paper gives 
technical ``blocks'' which allow one to establish 
the diffusion-scale tightness as long as the following two
properties hold:\\ 
(a) Global stability on the fluid-scale ($r$-scale), i.e.
convergence of fluid-scaled trajectories to the equilibrium point 
(plus an additional, related property);\\
(b) Local stability of the linear system in the neighborhood 
of the equilibrium point, i.e. the drift matrix of the limiting diffusion
process has all eigenvalues with negative real parts.\\
Given properties (a) and (b), our approach is to show tightness in several steps,
on the increasingly fine scales: fluid ($r$), then $r^{1/2+\epsilon}$, then 
diffusion ($r^{1/2}$) scale. 
We will make this discussion more specific in Section~\ref{sec-discussion}.

The distinctive feature of this approach, as opposed to most of the 
previous results on the diffusion-scale tightness for many-server models
(see \cite{Gamarnik_Momcilovic,Gamarnik_Stolyar,SY10}), is that
it does {\em not} rely on a single common Lyapunov function. 
(Finding/constructing a common Lyapunov function is usually a difficult task,
especially for the models with multiple server pools, like the one in this paper.)
We remind that in this paper we consider a system under strictly subcritical load,
and parts of our analysis do use this assumption.

\subsection{Brief literature review} 

A general overview of many-servers models, results and applications
to call centers can be found in \cite{GansKooleMandelbaum,AksinArmonyMehrotra}. For control policies for general models,
with multiple customer and server types, including blind policies, see e.g.
\cite{Atar2009,Armony_Ward,Gurvich_Whitt,Stolyar_Tezcan,Stolyar_Tezcan_underload,SY10,SY2012}
and references therein. Overviews of diffusion scale tightness (and limit interchange) results for single-pool
models in the many-servers Halfin-Whitt regime can be found, e.g., in \cite{Gamarnik_Momcilovic,GamarnikGoldbergGGN,Gamarnik_Stolyar}.
The diffusion scale tightness for the LQFS-LB policy, with the tree assumption and additionally assuming 
that the service rate (if non-zero) depends only on the server type, was proved in \cite{SY10}. The results in \cite{Gamarnik_Momcilovic,Gamarnik_Stolyar,SY10} use 
a common Lyapunov function; the work \cite{GamarnikGoldbergGGN}
does {\em not} use a Lyapunov function -- it relies instead on a
sample-path monotonicity/majorization property for a single-pool 
system under first-come-first-serve discipline.

\subsection{Layout of the rest of the paper} 
The model and the main result are given in Sections~\ref{section:model} and \ref{sec-main-result}, respectively. Section~\ref{sec-main-proof} contains the proofs. 
In concluding Section~\ref{sec-discussion} we discuss the results and technique.

\section{Model}
\label{section:model}

The model we consider is same as that in \cite{SY2012}. To improve self-containment
of this paper, we repeat the necessary definitions in this section.

\subsection{The model; Static Planning Problem}
\label{sec-SPP}

Consider the system in which there are $I$ customer classes, labeled $1,2,\dotsc,I$, and $J$ server pools, labeled $1,2,\dotsc,J$. 
(Servers within pool $j$ are referred to as class $j$ servers. Also,
throughout this paper the terms ``class'' and ``type'' are used interchangeably.)
The sets of customer classes and server pools will be denoted by $\CI$ and $\CJJJ$, respectively. We will use the indices $i$, $i'$ to refer to customer classes, and $j$, $j'$ to refer to server pools.

We are interested in the scaling properties of the system as it grows large. 
Namely, we consider a sequence of systems indexed by a scaling parameter $r$. As $r$ grows, the arrival rates and the sizes of the service pools, but not the speed of service, increase. Specifically, in the $r$th system, customers of type $i$ enter the system as a Poisson process of rate $\lambda_i r$, 
while the $j$th server pool has $\beta_j r$ individual servers. (All $\lambda_i$ and $\beta_j$ are positive parameters.) Customers may be accepted for service immediately upon arrival, or enter a queue; there is a separate queue for each customer type. Customers do not abandon the system. When a customer of type $i$ is accepted for service by a server in pool $j$, the service time is exponential of rate $\mu_{ij}$; the service rate depends both on the customer type and the server type, but \emph{not} on the scaling parameter $r$. If customers of type $i$ cannot be served by servers of class $j$, the service rate is $\mu_{ij} = 0$.

\begin{rem}
Strictly speaking, the quantity $\beta_j r$ may not be an integer, so we should define the number of servers in pool $j$ as, say, $\floor{\beta_j r}$. However, the change is not substantial, and will only unnecessarily complicate the notation.
\end{rem}

Consider the following
{\em static planning problem} (SPP):
\begin{subequations}\label{eqn:Static LP}
\begin{equation}
\min_{\lambda_{ij}^\circ, \rho} \rho,
\end{equation}
subject to
\beql{lp-constraint0}
\lambda_{ij}^\circ \geq 0, ~~\forall i,j
\end{equation}
\beql{lp-constraint1}
\sum_j \lambda_{ij}^\circ = \lambda_i , ~~\forall i
\end{equation}
\beql{lp-constraint2}
\sum_i \lambda_{ij}^\circ / (\beta_j \mu_{ij}) \leq \rho , ~~\forall j.
\end{equation}
\end{subequations}

Throughout this paper we will always make the following two assumptions about the solution to the SPP \eqref{eqn:Static LP}:
\begin{assumpt}[Complete resource pooling]
\label{assuption: CRP}
The SPP \eqref{eqn:Static LP} has a unique optimal solution $\{\lambda_{ij}^\circ, ~i\in \CI, ~j\in \CJJJ\}, \rho$. Define the \emph{basic activities} to be the pairs, or edges, $(ij)$ for which $\lambda_{ij}^\circ > 0$. Let $\CE$ be the set of basic activities. We further assume that the unique optimal solution is such that $\CE$ forms a  tree in the (undirected) graph with vertices set $\CI \cup \CJJJ$.
\end{assumpt}

\begin{assumpt}[Strictly subcritical load]
\label{CRP:2} The optimal solution to \eqref{eqn:Static LP} has $\rho < 1$.
\end{assumpt}

\begin{rem}
Assumption~\ref{assuption: CRP} is the \emph{complete resource pooling} (CRP) condition, which holds ``generically'' in a certain sense; see \cite[Theorem 2.2]{Stolyar_Tezcan}. Assumption~\ref{CRP:2} is essential for the main result of the paper.
\end{rem}

We assume that the basic activity tree is known in advance, and restrict our attention to the basic activities only. Namely, we assume that a type $i$ customer service in pool $j$ is allowed only if $(ij)\in\CE$. (Equivalently, we can a priori assume that $\CE$ is the set of {\em all} possible activities, i.e. $\mu_{ij} = 0$ when $(ij)\not\in\CE$, and $\CE$ is a tree. In this case CRP requires that all feasible activities are basic.) 
For a customer type $i$, let $\CSSS(i) = \{j: (ij)\in\CE \}$; for a server type $j$, let $\CC(j) = \{i: (ij)\in\CE \}$.

\subsection{Leaf activity priority (LAP) policy}
\label{section:LAP}

We analyze the performance of the following policy, which we call {\em leaf activity priority} (LAP).
The first step in its definition is the assignment of priorities to customer classes and activities.

Consider the basic activity tree, and assign priorities to the edges as follows. First, we assign priorities to customer classes by iterating the following procedure:
\begin{enumerate}
\item Pick a leaf of the tree;
\item If it is a customer class (rather than a server class), assign to it the highest priority that hasn't yet been assigned;
\item Remove the leaf from the tree.
\end{enumerate}
Without loss of generality, we assume the customer classes are numbered in order of priority (with 1 being highest). We now assign priorities to the edges of the basic activity tree by iterating the following procedure:
\begin{enumerate}
\item Pick the highest-priority customer class;
\item If this customer class {\em is} a leaf, pick the edge going out of it,
assign this edge the highest priority that hasn't yet been assigned, and remove the edge together with the customer class;
\item If this customer class is {\em not} a leaf, then pick any edge from it to a server class leaf (such necessarily exists),
assign to this edge the highest priority that hasn't yet been assigned, and remove the edge.
\end{enumerate}
It is not hard to verify that this algorithm will successfully assign priorities to all edges; it suffices to check that at any time the highest remaining priority
customer class will have at most one outgoing edge
to a non-leaf server class.

\begin{rem}\label{rem: priorities}
This algorithm does {\em not} produce a unique assignment of priorities, neither for the customer classes nor for the activities,
because there may be multiple options for picking a next leaf or edge to remove, in the corresponding procedures.
This is not a problem, because our results hold for {\em any} such assignment. 
Different priority assignments  may correspond to different equilibrium points (defined below in Section~\ref{sec: LAP equilibrium}); 
once we have picked a particular priority assignment, there is a (unique) corresponding equilibrium point, 
and we will be showing steady-state tightness around that point. 
Furthermore, the flexibility in assigning priorities may be a useful feature in practice. For example, it is easy to specialize the above
priority assignment procedure so that the lowest priority is given to any a priori picked activity. 
\end{rem}

We will write $(ij) < (i'j')$ to mean that activity $(ij)$ has higher priority than activity $(i'j')$. 
It follows from the priority assignment algorithm that 
$i < i'$ (customer class $i$ has higher priority than $i'$) implies $(ij) < (i'j')$. In particular,  if $j=j'$, we have $(ij) < (i'j)$ if and only if $i < i'$. 
Without loss of generality, we shall assume that the server classes are numbered so that the lowest-priority activity is $(IJ)$. 

Now we define the LAP policy itself. It consists of two parts: routing and scheduling. ``Routing'' determines where an arriving customer goes if it sees available servers of several different types. ``Scheduling'' determines which waiting customer a server picks if it sees customers of several different types waiting in queue.

{\bf Routing:} An arriving customer of type $i$ picks an unoccupied server in the pool $j \in \CSSS(i)$ such that $(ij) \leq (ij')$ for all $j' \in \CSSS(i)$ with idle servers. If no server pools in $\CSSS(i)$ have idle servers, the customer queues.

{\bf Scheduling:} A server of type $j$ upon completing a service picks the customer from the queue of type $i \in \CC(j)$ such that $i \leq i'$ for all $i' \in \CSSS(i)$ with $Q_{i'} > 0$. If no customer types in $\CC(j)$ have queues, the server remains idle.

We introduce the following notation (for the system with scaling parameter $r$): 
$\Psi^r_{ij}(t)$, the number of servers of type $j$ serving customers of type $i$ at time $t$; 
$Q^r_i(t)$, the number of customers of type $i$ waiting for service at time $t$.

Given the system operates under the LAP policy, the process\\
$\left((\Psi^r_{ij}(t),~(ij)\in \CE),(Q^r_i(t),~i \in \CI)\right),~t\ge 0,$\\
is a Markov process with countable state space. 

There are some obvious relations between system variables,
which hold for each process realization: 
for example, for any $j\in \CSSS(i)$ and any time $t$, either
$Q^r_i(t) = 0$ or $\sum_{i'} \Psi^r_{i'j}(t) = \beta_j r$; and so on.

\subsection{LAP equilibrium point}
\label{sec: LAP equilibrium}

Informally speaking, the equilibrium point 
$\left((\psi^*_{ij},~(ij) \in \CE),(q^*_i,~i \in \CI)\right)$
is the desired operating point for the (fluid scaled)
vector $\left((\Psi^r_{ij}/r,~(ij) \in \CE),(Q^r_i/r,~i \in \CI)\right)$ 
of occupancies and queue lengths under the LAP policy. 
The formal definition is given below.

Let us recursively define the quantities $\lambda_{ij}\ge 0$, which have the meaning of routing rates, scaled down by factor $1/r$. 
(These $\lambda_{ij}$ are {\em not} equal to the $\lambda_{ij}^\circ$ which comprise
 the optimal solution to the SPP \eqn{eqn:Static LP}.) For the activity $(1j)$ with the highest priority, define either $\lambda_{1j} = \lambda_1$ and $\psi^*_{1j} = \frac{\lambda_1}{\mu_{1j}}$, or $\psi^*_{1j} = \beta_j$ and $\lambda_{1j} = \beta_j \mu_{1j}$, according to whichever is smaller. Replace $\lambda_1$ by $\lambda_1 - \lambda_{1j}$ and $\beta_j$ by $\beta_j - \psi_{1j}^*$, and remove the edge $(1j)$ from the tree. We now proceed similarly with the remaining activities.

Formally, set
\[
\lambda_{ij} = \min\left(\lambda_i - \sum_{j': (ij') < (ij)} \lambda_{ij'}, \mu_{ij}\left(\beta_j - \sum_{i'<i} \frac{\lambda_{i'j}}{\mu_{i'j}}\right)\right).
\]
Since the definition is in terms of higher-priority activities, this defines the 
$(\lambda_{ij},~(ij) \in \CE)$ uniquely. The LAP equilibrium point is defined to be the vector
\[
\left((\psi^*_{ij},~(ij) \in \CE),(q^*_i,~i \in \CI)\right)
\]
given by
\beql{eqn:LAP equilibrium}
\psi^*_{ij} = \frac{\lambda_{ij}}{\mu_{ij}}, \quad q^*_i = 0 \text{ for all $(ij) \in \CE$, $i \in \CI$}.
\end{equation}
Clearly, by the above construction, we have
$$
\lambda_i = \sum_j \lambda_{ij} = \sum_j \mu_{ij} \psi_{ij}^*, ~~i\in \CI, ~~~
\sum_i \psi_{ij}^* \leq \beta_j, ~~j\in\CJJJ.
$$

To avoid trivial complications,  throughout the paper we make the following assumption:
\begin{assumpt}\label{ass:rho}
If $(\psi_{ij},~(ij) \in \CE)$ are such that $\psi_{ij} \geq 0$, $\lambda_i = \sum_j \mu_{ij} \psi_{ij}$, and $\sum_i \psi_{ij} \leq \beta_j$ for all $j$, then 
$\psi_{ij} > 0$ for all $(ij)\in \CE$.
\end{assumpt}
This assumption implies, in particular, that for the equilibrium point we must have $\psi_{ij}^* > 0$ for all $(ij)\in \CE$ and, moreover,
$\sum_i \psi^*_{ij} = \beta_j$ for all $j < J$ and $\sum_i \psi^*_{iJ} < \beta_J$.

The Assumption~\ref{ass:rho} means that the system needs to employ (on average) all activities in $\CE$ in order to be able to handle the load.
It holds, for example, whenever $\rho$ is sufficiently close to 1.

\begin{rem}
Assumption~\ref{ass:rho} is technical. Our main result, the diffusion-scale tightness in 
Theorem~\ref{th-stationary-scale-main-strong}, can be proved without it, by following the approach presented in the paper.
But, it simplifies the statements and proofs 
of many auxiliary results, and thus substantially improves the exposition.
\end{rem}

\subsection{Discussion of LAP policy features and assumptions}
\label{sec-lap-discussion}

The starting point in the definition of LAP is a fixed set of allowed 
activities $\CE$, and the assumption that it forms a tree.
How the tree $\CE$ is determined is, in a sense, a secondary question.
For example, the structure of the system itself may be such that
the set of {\em all} possible activities is a tree $\CE$. 
If not, $\CE$ can be computed as a set of basic activities of the
static planning problem (SPP) \eqn{eqn:Static LP}. Solving SPP \eqn{eqn:Static LP},
of course, requires the knowledge of parameters $\lambda_i$ and $\mu_{ij}$.
Note, however, that, typically (in the sense 
specified in \cite[Theorem 2.2]{Stolyar_Tezcan}),
small perturbation of parameters 
$\lambda_i$ and $\mu_{ij}$, while changing the LP solution, will {\em not} change
the set of basic activities. Therefore, computing $\CE$ by solving 
SSP \eqn{eqn:Static LP} does {\em not} require {\em exact knowledge}
of the system parameters, and in many cases approximate knowledge 
of the parameters may well be enough to find the ``correct'' set $\CE$.

A typical solution of SPP \eqn{eqn:Static LP} is such that the set of basic activities 
$\CE$  forms a forest (graph without cycles), not necessarily a tree (which is a connected
forest); moreover, within each tree-component of the forest 
the complete resource pooling (CRP) condition will hold.
(Again, see \cite[Theorem 2.2]{Stolyar_Tezcan}.) In this case, the LAP 
algorithm can be applied to each of the tree-components separately.

Finally, we emphasize that while the objective of SPP \eqn{eqn:Static LP}
is load balancing, the LAP algorithm does {\em not} try to balance load
of the server pools.
(Hence the values of $\lambda_{ij}$ that define the equilibrium point in 
Section~\ref{sec: LAP equilibrium} are {\em not} equal to the values
$\lambda_{ij}^\circ$ solving \eqn{eqn:Static LP}.)
Instead of balancing load, LAP algorithm greedily tries to ``pack'' 
customers into pools according to activity priorities. As a result, the equilibrium 
point is such that some of the pools are completely ``packed'',
while other pools (exactly one under simplifying technical Assumption~\ref{ass:rho})
have non-zero fraction of idle servers.

\subsection{Basic notation}

Vector $(\xi_i, ~i\in \CI)$, where $\xi$ can be any symbol, is often written as $(\xi_i)$;
similarly, 
$(\xi_j, ~j\in \CJJJ)=(\xi_j)$
and 
$(\xi_{ij}, ~(ij)\in \CE)=(\xi_{ij})$. 
Furthermore, we often use notation $(\eta_{ij},\xi_i)$   to mean  
$((\eta_{ij}, ~(ij)\in \CE), (\xi_i, ~i\in \CI))$,
and similar notations as well.
Unless specified otherwise, $\sum_i \xi_{ij} = \sum_{i \in \CC(j)} \xi_{ij}$ and $\sum_j \xi_{ij} = \sum_{j \in \CSSS(i)} \xi_{ij}$. For functions (or random processes) $(\xi(t), ~t\ge 0)$ we often write $\xi(\cdot)$. (And similarly for functions with domain different from $[0,\infty)$.) So, for example, $(\xi_i(\cdot))$ signifies $((\xi_i(t), i\in \CI), ~t\ge 0)$. 

In the Euclidean space $\mathbb{R}^d$ 
(with appropriate dimension $d$): $\abs{x}$ denotes standard Euclidean
norm of vector $x$; 
symbol $\to$ denotes ordinary convergence; 
we write simply $0$ for a zero vector. Abbreviation {\em u.o.c.} means 
{\em uniform on compact sets} convergence of functions, 
with the domain defined explicitly or by the context.
We always consider the Borel $\sigma$-algebra on $\mathbb{R}^d$ 
 when it is viewed as a measurable space.
The symbol $\stackrel{w}{\rightarrow}$ denotes weak convergence 
of probability distributions.
 {\em W.p.1} means {\em with probability 1}. We will consider a sequence
of systems indexed by scaling parameter $r$ increasing to infinity, and will use abbreviation
{\em w.p.1-l.r} as a short for {\em w.p.1 for all sufficiently large $r$}.

We denote by $Dist[\xi]$ the distribution of a random element $\xi$, and
by $Inv[\xi(\cdot)]$ the stationary distribution of a Markov process
$\xi(\cdot)$ (it will be unique in all cases that we consider).

\section{Main result}
\label{sec-main-result}

It was shown in  \cite[Theorem 10]{SY2012} that, if the system 
under LAP policy
is strictly subcritically loaded, i.e. $\rho<1$,
then for all large $r$ the Markov process $(\Psi^r_{ij}(\cdot), Q^r_i(\cdot))$ 
is positive recurrent,
has unique stationary distribution\\ $Inv[(\Psi^r_{ij}(\cdot), Q^r_i(\cdot))]$
and, moreover,
the sequence of stationary distributions is tight on the scale $r^{1/2+\epsilon}$ with any $\epsilon>0$.
In this paper we strengthen this result by showing that the invariant distributions are in fact tight 
on the diffusion, i.e. $r^{1/2}$, scale. This is, of course, 
 the strongest possible tightness result for 
the system and the asymptotic regime in this paper. As a consequence, we obtain a limit interchange result: the 
limit of diffusion-scaled invariant distributions is equal to the invariant distribution of the limiting diffusion process.

Denote by
$Z^r_j(t)=\sum_i \Psi^r_{ij}(t) - r\sum_i \psi^*_{ij}$ the ``idleness''
of pool $j$. 
Recall that for each $j<J$, $\sum_i \psi^*_{ij}=\beta_j$ and therefore $Z^r_j(t)\le 0$.
Let $L'$ be the linear mapping (defined in \cite[Section 5.2]{SY2012}),
which takes a vector $(\xi_i)$ with real components into
the vector $(\eta_{ij})$, uniquely solving 
\beql{eq-Lprime-def}
\sum_j \eta_{ij} = \xi_i, ~\forall i, ~~~~~~ \sum_i \eta_{ij} = 0, ~j < J.
\end{equation}

\begin{thm}
\label{th-stationary-scale-main-strong}
Consider the sequence of systems under LAP policy,
 in the scaling regime and under the assumptions
specified in Section~\ref {section:model}, with $\rho<1$.
Then, the sequence of diffusion-scaled stationary distributions,
$Inv[r^{-1/2} (\Psi^r_{ij}(\cdot)-\psi_{ij}^* r, Q^r_i(\cdot))]$, is tight.
Moreover, 
\beql{eq-main-res1}
Inv[r^{-1/2} (\Psi^r_{ij}(\cdot)-\psi_{ij}^* r)] \stackrel{w}{\rightarrow} 
Inv[(\breve \Psi_{ij}(\cdot))], ~~r\to \infty,
\end{equation}
where $(\breve \Psi_{ij}(\cdot))$ is 
the diffusion process, defined by the stochastic differential equation 
\beql{eq-diffusion-psi-generic-1}
d (\breve \Psi_{ij}(t)) =
L' d \left(\sqrt{\lambda_i} B^{(a)}_i(t)\right) 
- L' d\left(\sum_j \sqrt{\mu_{ij} \psi_{ij}^*} B^{(s)}_{ij}(t)\right) -
\end{equation}
$$
 L' \left(\sum_j \mu_{ij} \breve \Psi_{ij}(t)\right) dt,
$$
with all $B^{(a)}_i(\cdot)$ and $B^{(s)}_{ij}(\cdot)$ being independent
standard Brownian motions;\\
and for any $\nu>0$
\beql{eq-main-res2}
Inv[r^{-\nu} \left((Q^r_i(\cdot)),(Z^r_j(\cdot), ~j<J)\right)] 
\stackrel{w}{\rightarrow} Dist[0],  ~~r\to \infty,
\end{equation}
where $Dist[0]$ is the Dirac measure concentrated at the zero vector.
\end{thm}

\begin{rem}
\label{rem-QZ}
Property \eqn{eq-main-res2} shows that 
the distributions of 
all queue lengths and of the idlenesses
in pools $j\ne J$, are tight on the scale $r^{\nu}$ for 
any $\nu>0$. As we will see, this fact is an ``ingredient'' 
of the proof of diffusion-scale tightness and \eqn{eq-main-res1}.
Also, it is not surprizing and is a consequence 
of the priority discipline and (for the queues) of strict subcriticality,
$\rho<1$. As discussed in Section~\ref{sec-lap-discussion},
LAP tries to ``pack'' server pools according to the activity priority order.
As a result, when the idleness in a pool $j\ne J$ is non-zero, then, roughly speaking, 
the arrival rate into the pool exceeds the departure rate by a factor greater than $1$;
similarly, the departure rate from any non-zero queue exceeds the arrival rate by a factor greater than $1$.
Therefore, it is natural to expect that even the stronger property
than \eqn{eq-main-res2} holds, namely the sequence of unscaled stationary
distributions
$
Inv[\left((Q^r_i(\cdot)),(Z^r_j(\cdot), ~j<J)\right)] 
$
is tight. We do not pursue proving this fact in the paper, because 
establishing diffusion-scale tightness and \eqn{eq-main-res1} is our main goal.

\end{rem}

\section{Proof of Theorem~\ref{th-stationary-scale-main-strong}}
\label{sec-main-proof}

In the rest of the paper, we will use the following additional notation for the system variables.
For a system with parameter $r$, we denote:\\
$X^r_i(t)=\sum_j\Psi^r_{ij}(t)+Q^r_i(t)$ is the total number of customers of type $i$ in the system at time $t$;\\
$A^r_i(t)$ is the total number of customers of type $i$ exogenous arrivals into the system in interval $[0,t]$;\\
$D^r_{ij}(t)$ is the total number of customers of type $i$ that completed the service in pool $j$ (and departed the system) in interval $[0,t]$;\\
finally, we will use short notation $F^r(t)=(\Psi^r_{ij}(t)-\psi_{ij}^* r, Q^r_i(t))$.

We can and do assume that a random realization of the system with parameter $r$ is determined by its initial state and realizations of ``driving'' unit-rate, mutually independent, Poisson processes $\Pi_i^{(a)}(\cdot), i\in \CI,$ and 
$\Pi_{ij}^{(s)}(\cdot), (ij)\in \CE$, as follows:
$$
A_i^r(t) = \Pi_i^{(a)}(\lambda_i r t), ~~~~D_{ij}^r(t)=\Pi_{ij}^{(s)}\Bigl(\mu_{ij} \int_0^t \Psi^r_{ij}(u) du \Bigr);
$$
the driving Poisson processes are common for all $r$.
It is easy to see that, given the LAP policy, with probability 1
 the realizations of these driving processes (along with initial state)
indeed uniquely define the system process realization.

Finally, the diffusion scaled variables are defined as follows:
$$
(\hat \Psi^r_{ij}(t), \hat Q^r_i(t))=
r^{-1/2} (\Psi^r_{ij}(t)-\psi_{ij}^* r, Q^r_i(t)),
$$
$\hat X^r_i(t) = r^{-1/2} [X^r_i(t) - \sum_j \psi_{ij}^* r]$,
$\hat Z^r_j(t)=r^{-1/2} Z^r_j(t)$.

Throughout this section, we will use the following  
strong approximation of Poisson processes, 
available e.g. in \cite[Chapters 1 and 2]{Csorgo_Horvath}: 

\begin{prop}\label{thm:strong approximation-111-clean}
A unit rate Poisson process $\Pi(\cdot)$ and 
a standard Brownian motion $W(\cdot)$ can be constructed on a common
probability space in such a way that the following holds
for some fixed positive constants $C_1$, $C_2$, $C_3$:
$\forall T>1$ and
$\forall u \geq 0$
\[
\BP\left(\sup_{0 \leq t \leq T} \abs{\Pi(t) - t - W(t)} \geq C_1 \log T + u\right) \leq C_2 e^{-C_3 u}.
\]
\end{prop}

We will also need the following form of a functional strong law 
of large numbers for a Poisson process. It is obtained using standard large deviations
estimates, e.g. analogously to the way it is done
 in the proof of \cite[Lemma 4.3]{ShSt2000}. 

\begin{prop}
\label{prop:ShSt}
For a unit rate Poisson process $\Pi(\cdot)$, the following holds with probability $1$.
For any $\nu \in (0,1)$ and any $c>1$, uniformly in $t_1,t_2 \in [0,r^c]$
such that $t_2-t_1 \ge r^{\nu}$,
$$
[\Pi(t_2)-\Pi(t_1)]/[t_2-t_1]\to 1,  ~~r\to \infty.
$$
\end{prop}

Throughout this paper, we will use Proposition~\ref{prop:ShSt} with arbitrary
fixed $c>1$: this ensures that for any fixed $T>0$, the interval
$[0,Tr\log r]$ is contained within $[0,r^c]$ for all large $r$.
Proposition~\ref{prop:ShSt}, in particular, immediately implies the following upper bound 
on the rate at which system variables can change.
There exists $C>0$, such that for any $\nu \in (0,1)$ and any $\alpha>0$,
w.p.1-l.r, uniformly in $t_1,t_2 \in [0,r^{c-1}]$
such that $t_2-t_1 \ge \alpha r^{\nu}/r$,
\beql{eq-simple-bound}
\max_{t\in[t_1,t_2]} |Q_i^r(t)-Q_i^r(t_1)| < C (t_2-t_1)r, ~\forall i,
\end{equation}
and similarly for $\Psi_{ij}^r(\cdot), \forall (ij)$, $Z_j^r(\cdot), \forall j$, and 
$F^r(\cdot)$. Indeed, in a system with parameter $r$, the customer 
arrival and departure events occur, ``at most'',  as\\
$\Pi\left( [\sum_i \lambda_i +(\sum_j \beta_j)\max_{(ij)} \mu_{ij}]r\right)$,
where $\Pi(\cdot)$ is a unit rate Poisson
process; therefore, the condition
$t_2-t_1 \ge \alpha r^{\nu}/r$ in the $r$-th system
guarantees that the interval $[t_1,t_2]$ corresponds to at least
$O([t_2-t_1]r)= O(r^{\nu})$-long time interval for $\Pi(\cdot)$, and then 
Proposition~\ref{prop:ShSt} applies.

\begin{lem}
\label{lem-hit-sqrt-r}
There exists $T>0$ such that for any $\epsilon \in (0,1/2)$ 
the following holds.
For any  $\delta>0$, 
there exists a sufficiently large $C_7>0$ such that, 
uniformly on all sufficiently large $r$ and all $|F^r(0)| \le g(r)=r^{1/2 + \epsilon}$,
the probability of $|F^r(t)| \le C_7 r^{1/2}$ occurring within 
$[0,\epsilon T \log r]$ is at least $1-\delta$.
\end{lem}

\begin{proof}
The proof is by contradiction.
If lemma does not hold, then there exists
a function $g_*(r)$ such that
$g_*(r)/r^{1/2}\uparrow\infty$ and the probability 
of starting from $|F^r(0)| \le g(r)$ and not hitting 
$|F^r(t)| \le g_*(r)$ within time $\epsilon T \log r$, does not vanish.
We will prove that it has to vanish, thus establishing a contradiction. 

Denote $|F^r(0)|=h(r)$. We now specify the choice of $T$.
We note that all results in Sections 5.2-5.3  of \cite{SY2012}, concerning hydrodynamic and local-fluid limits,
hold as is for any function $h(r)$ such that $h(r)/r^{1/2}\to\infty$. 
(The condition $h(r) \ge r^{1/2+\epsilon}$ was used in \cite{SY2012} only when 
the results of Sections 5.2-5.3 there were {\em applied}.)
Then, by Corollary 25 and condition (23)
in \cite{SY2012}, we can and do choose
a sufficiently large $T>0$ such that the conditions
\beql{eq-poission-sample-paths777}
\max_{t\in[0,T]} | \Pi_i^{(a)}(\lambda_i rt) - \lambda_i rt | \le \delta_2 h(r), ~\forall i,
\end{equation}
and similar for $\Pi_{ij}^{(s)}, ~\forall (ij)$, 
 with sufficiently small fixed $\delta_2>0$, guarantee that 
condition $g(r) \ge h(r)=|F^r(0)|\ge g_*(r)$
implies that $|F^r|$ decreases at least by a factor $K>1$ in $[0,T]$.
Let us see how the probability of \eqn{eq-poission-sample-paths777}
depends on $h(r)$, or more conveniently on $h_1(r)=h(r)/r^{1/2}$.
(Note that $h_1(r)\uparrow\infty$ when $h(r)\ge g_*(r)$.)

Now we will use Proposition~\ref{thm:strong approximation-111-clean}.
In its statement let us replace $\Pi$ with $\Pi_i^{(a)}$,
$t$ with $\lambda_i rt$, $T$ with $\lambda_i rT$, 
make $u$ a function of $r$,
say $u=r^{1/4}$. Then, with probability at least $1- C_2 e^{-C_3 r^{1/4}}$,
$$
\BP\left\{   \max_{t\in[0,T]}  | \Pi_i^{(a)}(\lambda_i rt) - \lambda_i rt | \le  \max_{t\in[0,T]}  |W(\lambda_i rt)|    + C_1 \log (\lambda_i rT)
+ r^{1/4}
    \right\} \ge 
$$
$$
1- C_2 e^{-C_3 r^{1/4}},
$$
where $C_1, C_2,C_3$ are universal constants (from the statement of Proposition~\ref{thm:strong approximation-111-clean}).
Next, observe that $(W(\lambda_i rt)/h(r), ~t\ge 0)$, where $W(\cdot)$
is a standard Brownian motion, is equal in distribution to
$(\sqrt{\lambda_i}W(t)/h_1(r), ~t\ge 0)$.
Therefore, 
$$
\BP\left\{   \max_{t\in [0,T]}  |W(\lambda_i rt)|    \le (\delta_2/2) h(r) \right\} \ge 1- C_4 e^{-C_5 (h_1(r))^2},
$$
where positive constants $C_4,C_5$ depend on 
$\delta_2$ and $T$ (and system parameters).
We conclude that
the probability of \eqn{eq-poission-sample-paths777} is lower bounded
by
$$
1- C_2 e^{-C_3 r^{1/4}} - C_4 e^{-C_5 (h_1(r))^2}.
$$
Denote
$$
p_i = \BP\{|F^r(t)| \le g_*(r) ~\mbox{for some}~t\in [0, iT]~|~ 
|F^r(0)| \le K^i g_*(r)\}, ~~~i=0,1,2,\ldots
$$
We can write, for any $i\ge 1$,
$$
p_i \ge [1 - C_2 e^{-C_3 r^{1/4}} - C_4 \exp\{-C_5K^{2i}(g_*(r)/r^{1/2})^2\}] p_{i-1}.
$$
We are interested in $p_k$ with $k=\epsilon \log r$, which is
lower bounded as
$$
p_k \ge \prod_{i=1}^k [1 - C_2 e^{-C_3r^{1/4}} - C_4 \exp\{-C_5 K^{2i} g_*(r)^2/r\}]
\ge 
$$
$$
1 - \sum_{i=1}^k [C_2 e^{-C_3r^{1/4}} + C_4 \exp\{-C_5 K^{2i} g_*(r)^2/r\}].
$$
The sum vanishes as $r\to\infty$, and so is $1-p_k$.
\end{proof}

The key part of the rest of the proof of 
Theorem~\ref{th-stationary-scale-main-strong}, is
to show that, informally speaking, if the process ``hits'' 
the set $\{|F^r| \le C_7 r^{1/2}\}$ anywhere within $[0,\epsilon T \log r]$,
then it stays ``on $r^{1/2}$-scale'' at time $\epsilon T \log r$ as well.
To do this we will exploit the closeness
of the diffusion scaled process to the diffusion limit, 
on a $\epsilon T \log r$-long interval (i.e., with length increasing with $r$),
when $\epsilon$ is small enough. This will be formalized 
in Lemma~\ref{lem-close-to-diffusion}, but to apply it we need an additional step, given by the following

\begin{lem}
\label{lem-down-to-subdiffusion-scale}
There exist $T_8>0$ and $C_8>0$ such that the following holds.
For any fixed $C_9>0$, $\delta_9>0$ and $\nu_9\in (0,1/2)$, 
uniformly on initial states
$|F^r(0)|\le C_{9} r^{1/2}$, as $r\to\infty$,
\beql{eq-11111}
\BP\{\max_{t\in [0,T_8 C_9 r^{-1/2}]} |F^r(t)|\le C_8 C_{9} r^{1/2}\} \to 1,
\end{equation}
\beql{eq-22222}
\BP\{\exists t\in [0,T_8 C_9 r^{-1/2}]:~
|(Q_i^r(t))|+ |(Z_j^r(t), j<J)| \le \delta_9 r^{\nu_9}\} \to 1.
\end{equation}
\end{lem}

We will use this lemma (and Lemma~\ref{lem-sup-tight-very} below)
with $0<\nu_9<1/4$. 

\begin{proof}

Let us first discuss the basic intuition behind the result, which is extremely simple,
and will be useful not only for this proof, but for some other proofs in the paper
as well. Within a fixed $O(r^{-1/2})$ time, $F^r(t)$ can change at most by $O(r^{1/2})$ -- 
see \eqn{eq-simple-bound} -- and therefore, for all $(ij)$,
$\Psi^r_{ij}(t)/[\psi_{ij}^* r] \approx 1$
holds. Now, consider the highest-priority activity $(1j)$. Suppose
customer class $1$ is a leaf. Then, there must exist at least one other
activity $(ij)$, associated with the same pool $j$. The arrival rate of
type $1$ is $\lambda_1 r = \mu_{1j}\psi_{1j}^* r$, while the total service completion 
rate at pool $j$ is at least $\mu_{1j}\Psi^r_{1j}(t)+\mu_{ij}\Psi^r_{ij}(t)
\approx \mu_{1j}\psi_{1j}^* r + \mu_{ij}\psi_{ij}^* r = \lambda_1 r + \mu_{ij}\psi_{ij}^* r$.
This means that, since type $1$ has the highest priority at pool $j$, the queue
 $Q_1^r(t)$, when non-zero, ``drains'' at the rate
at least $O(r)$, ``hits'' $r^{\nu_9}$ scale
within $O(r^{-1/2})$ time and ``stays there.''
Suppose now that class $1$ is not a leaf. Then pool $j$ must be a leaf,
i.e. it serves type $1$ exclusively, $\psi_{1j}^*=\beta_j$,
and there must be at least one
other activity $(1m)$, associated with type $1$, implying $\lambda_1 \ge 
\mu_{1j}\psi_{1j}^* + \mu_{1m}\psi_{1m}^* > \mu_{1j}\beta_j$.
The difference between type $1$ arrival rate and the rate they are served 
by pool $j$ is at least $[\lambda_1 - \mu_{1j}\beta_j]r=O(r)$.
This means that the idleness $|Z_j^r(t)|$, when non-zero, decreases at the rate
at least $O(r)$, ``hits'' $r^{\nu}$ scale within $O(r^{-1/2})$ time and ``stays there.''
We ``remove'' activity $(1j)$ from the activity tree.
The argument proceeds by considering all activities $(ij)$ in sequence,
from the highest to lowest priority; at each step either $Q_i^r(t)$
or $Z_j^r(t)$ is ``eliminated'', depending on $i$ or $j$, respectively,
being a leaf of the current activity tree. The exception is when $j=J$ is the pool
serving the lowest priority activity $(IJ)$: in this case $Z_J^r(t)$
is {\em not} eliminated. We now proceed with a sketch of a formal
argument -- details can be easily ``recovered'' by the reader.

The proof of \eqn{eq-11111} is an immediate consequence of \eqn{eq-simple-bound}.
Indeed, for any $T_8>0$, w.p.1-l.r,
 the value of $|F^r(t)-F^r(0)|$ with $t\in [0,T_8 C_9 r^{-1/2}]$
is upper bounded by $C T_8 C_9 r^{1/2}$. So, for any chosen $T_8$ we can 
choose $C_8 > 1+ C T_8$. 

Property \eqn{eq-11111}, in particular, means
that for any fixed $T_8>0$, w.p.1, for any $(ij)\in \CE$,
uniformly in $t\in [0,T_8 C_9 r^{-1/2}]$ we have 
\beql{eq-t1}
\Psi^r_{ij}(t)/[\psi_{ij}^* r] \to 1.
\end{equation}

To prove \eqn{eq-22222}, we consider and ``eliminate''
activities one by one, in the order 
of their priority. The choice of $T_8$ will be made later -- for now it
is a fixed constant, and we consider the process in the interval 
$[0,T_8 C_9 r^{-1/2}]$.
We start with the highest priority activity $(1j)$.
Suppose first that customer class $1$ is a leaf of the activity
tree. (In this case, $\CC(j)$ necessarily contains at least
one customer class in addition to $1$.)
Consider any $0<C_1< \sum_{i\ne 1} \mu_{ij} \psi_{ij}^*$.
Then, for any $\delta>0$, there exists a sufficiently small $\delta_1>0$,
such that, w.p.1-l.r, uniformly in $t\in  [0,T_8 C_9 r^{-1/2}]$,
condition $Q^r_1(t) \ge \delta r^{\nu_9}$ implies 
$Q^r_1(t+\delta_1 r^{\nu_9}/r)-Q^r_1(t) < - C_1 \delta_1 r^{\nu_9}$
(because {\em all} departures from pool $j$ are replaced by class $1$ customers from the queue),
and for any $Q^r_1(t)$ we have (by \eqn{eq-simple-bound}) 
$\max_{\tau\in[0,\delta_1 r^{\nu_9}/r]} Q^r_1(t+\tau) < Q^r_1(t) + C \delta_1 r^{\nu_9}$.
This means that w.p.1. 
$$
\max_{t\in [T', T_8 C_9 r^{-1/2}]} Q^r_1(t) \le (\delta+C\delta_1) r^{\nu_9},
$$
where $T'=2 (1/C_1) C_9 r^{-1/2}$.
Note that this holds for any $\delta$ and the corresponding $\delta_1$,
both of which can be chosen arbitrarily small. We conclude that w.p.1. 
\beql{eq-t2}
\max_{t\in [T', T_8 C_9 r^{-1/2}]} Q^r_1(t)/r^{\nu_9} \to 0.
\end{equation}
This means, in particular, that in $ [T', T_8 C_9 r^{-1/2}]$, 
the number of exogenous class $1$ arrivals matches the number of class $1$ customers entering service,
up to  $o(r^{\nu_9})$ quantities. Formally, the following holds.
Denote by $\Xi_{ij}^r(t_1,t_2)$ 
the number of type $i$ customers that enter service in pool $j$
in the time interval $(t_1,t_2]$. For any fixed $\delta_1>0$,
w.p.1, uniformly in $t_1,t_2\in [T', T_8 C_9 r^{-1/2}]$
such that $t_2-t_1\ge \delta_1 r^{\nu_9}/r$,
\beql{eq-t3}
\Xi_{1j}^r(t_1,t_2)/[\lambda_1 r (t_2-t_1)] \to 1.
\end{equation}
Finally, note that, again by \eqn{eq-simple-bound}, w.p.1-l.r,
at time $T'$, $|F^r|$ is at most by a constant factor 
(depending on $C_1$) larger than $C_{9} r^{1/2}$. 
Our conclusions about the $(1j)$ activity can be informally summarized
as follows:  within a time $T'=2 (1/C_1) C_9 r^{-1/2}$, proportional
to $C_9 r^{-1/2}$, the value of $Q^r_1(t)/r^{\nu_9}$ ``drains to $0$'' 
and ``stays there'' (in the sense of \eqn{eq-t2})
until the end of interval $[0,T_8 C_9 r^{-1/2}]$;
moreover, in the interval $[T',T_8 C_9 r^{-1/2}]$, the rate at which server pool $j$ ``takes''
type $1$ customers is ``equal'' (in the sense of \eqn{eq-t3})
to their arrival rate $\lambda_1 r$.
Therefore, starting time $T'$ we can ``eliminate'' and ``ignore'' activity $(1j)$ in the
sense that we know that 
the rate at which pool $j$ can take for service 
customers of the types other than $1$ is ``at least'' $[\sum_{i\ne 1} \mu_{ij} \psi_{ij}^*]r$.
More precisely, if we denote by $S^r_{(\ne 1),j}(t_1,t_2)$ the number of times 
in the interval $(t_1,t_2]$ when a service completion by a server in pool $j$ was
{\em not} followed (either immediately or after some idle period)
by taking a type $1$ customer for service, then the following holds:
for any fixed $\delta_1>0$,
w.p.1, uniformly in $t_1,t_2\in [T', T_8 C_9 r^{-1/2}]$
such that $t_2-t_1\ge \delta_1 r^{\nu_9}/r$,
\beql{eq-t4}
\frac{S^r_{(\ne 1),j}(t_1,t_2)}{[\sum_{i\ne 1} \mu_{ij} \psi_{ij}^*]r(t_2-t_1)} \to 1.
\end{equation}
Moreover, $|F^r(T')|$ is at most by a factor 
larger than $C_{9} r^{1/2}$, which is the upper bound on $|F^r(0)|$.

Suppose now that class $1$ is not a leaf. Then necessarily poll $j$ is a leaf
and $j<J$.
In this case, by looking at the evolution of idleness $Z^r_j(t)$,
and using similar arguments, we can show that, again,  within a time proportional
to $C_9 r^{-1/2}$, let us call it $T''$,
the value of $Z^r_j(t)/r^{\nu_9}$ ``drains to $0$'' 
and ``stays there'' (in the sense analogous to \eqn{eq-t2})
until the end of interval $[0,T_8 C_9 r^{-1/2}]$;
this in turn means that the rate at which type $1$ customers will 
enter pool $j$ in the interval $[T'',T_8 C_9 r^{-1/2}]$
will be ``equal'' (in the sense analogous to \eqn{eq-t3})
to $\mu_{1j} \beta_j r$.
And again, w.p.1-l.r,
$|F^r(T'')|$ is at most by a constant factor 
larger than $C_{9} r^{1/2}$. 
Therefore, starting time $T''$ we can ``eliminate'' activity $(1j)$
in the sense that we can ``ignore'' pool $j$ and ``assume'' that
the arrival rate of type $1$ customers in the rest of the system is ``equal'' 
to $\lambda_1 r - \mu_{1j} \beta_j r$. (The latter is in the sense analogous to \eqn{eq-t4},
but where we count the type $1$ arrivals that were {\em not} taken for service
in the corresponding interval $(t_1,t_2]$.)

We can proceed to ``eliminate'' the second-highest priority activity, and so on.
The total time for all scaled queues $Q^r_i(t)/r^{\nu_9}$ and 
all idlenesses $Z^r_j(t)/r^{\nu_9}$, $j<J$, to ``drain to $0$'' will be 
proportional to $C_9 r^{-1/2}$, say  $T'_8 C_9 r^{-1/2}$. We then choose $T_8>T'_8$.
We omit further details, except to emphasize again that property \eqn{eq-22222}
does {\em not} include ``idleness'' $Z^r_J$ for the pool $J$ serving
the lowest-priority activity $(IJ)$.
\end{proof}

\begin{lem}
\label{lem-close-to-diffusion}
Let $T>0$ be fixed.
For a sufficiently small $\epsilon>0$ the following holds.
For any fixed $C_{11}>0$, $\delta_9>0$, and $\nu_9\in (0,1/4)$,
uniformly on initial states satisfying
$|F^r(0)|\le C_{11} r^{1/2}$ and
$|(Q_i^r(0), ~\forall i)|+ |(Z_j^r(0), j<J)| \le \delta_9 r^{\nu_9}$,
\beql{eq-close-to-diffusion}
\max_{t\in [0,\epsilon T \log r]} |(\hat \Psi_{ij}^r(t))-(\breve \Psi_{ij}^r(t))|
\implies 0,
\end{equation}
where $(\breve \Psi_{ij}^r(\cdot))$ is a (strongly) unique strong solution 
of the stochastic integral equation \eqn{eq-diffusion-psi}
(constructed on a common probability space with $(\hat \Psi_{ij}^r(\cdot))$),
with the initial state $(\breve \Psi_{ij}^r(0))= (\hat \Psi_{ij}^r(0))$.
\end{lem}

To prove this lemma we will need a series of auxiliary results.

\begin{lem}
\label{lem-sup-tight-very}
There exists $C_{10}>0$ such that the following holds for any $\epsilon>0$, $T>0$,
$C_{11}>0$, $\delta_9>0$ and $\nu_9\in (0,1/2)$.
As $r\to\infty$, uniformly on all initial states such that
$|F^r(0)|\le C_{11} r^{1/2}$ and
$|(Q_i^r(0))|+ |(Z_j^r(0), j<J)| \le \delta_9 r^{\nu_9}$,
we have
\beql{eq-44444}
\BP\{\max_{t\in [0,T \log r]} |F^r(t)|\le r^{1/2+\epsilon}\} \to 1,
\end{equation}
\beql{eq-55555}
\BP\{\max_{t\in [0,T \log r]} [|(Q_i^r(t))|+ |(Z_j^r(t), j<J)|]
\le C_{10}\delta_9 r^{\nu_9}\} \to 1.
\end{equation}
\end{lem}

\begin{proof} 
The proof of property \eqn{eq-44444} is already contained 
in the proof of \cite[Theorem 10(ii)]{SY2012}.
Indeed, that proof considers the process on the interval
$[0,T \log r]$ and shows that, starting with $|F^r(0)|=o(r)$,
w.p.1-l.r, $|F^r(t)|$ ``hits'' $r^{1/2+\epsilon}$-scale
somewhere within $[0,T \log r]$, and then ``stays'' on this scale
until the end of the interval. In our case, $|F^r(0)|$ is already
on the $r^{1/2+\epsilon}$-scale, and so the process w.p.1-l.r
stays in it in the entire interval $[0,T \log r]$.

Given \eqn{eq-44444}, to prove \eqn{eq-55555} we can ``reuse''
the proof of \eqn{eq-22222} of Lemma~\ref{lem-down-to-subdiffusion-scale}.
In that proof we showed that starting $|F^r(0)|=O(r^{1/2})$,
w.p.1-l.r, the quantity $[|(Q_i^r(t), ~\forall i)|+ |(Z_j^r(t), j<J)|]$
``hits $r^{\nu_9}$-scale '' within an $O(r^{-1/2})$-long time interval and ``stays there''
until the end of that time interval. (See \eqn{eq-t2}.)
In our case, the initial state 
is already such that $|(Q_i^r(0), ~\forall i)|+ |(Z_j^r(0), j<J)|=O(r^{\nu_9})$,
and therefore this quantity stays $O(r^{\nu_9})$ in the entire interval.
The fact that here we consider a much longer interval,
namely $O(\log r)$ as opposed to $O(r^{-1/2})$, is immaterial,
because \eqn{eq-44444}, and therefore \eqn{eq-t1},
holds on the entire interval and $r\log r = o(r^c)$
(so that we can use Proposition~\ref{prop:ShSt}).
We omit further details.
\end{proof}

\begin{prop}
\label{prop-strong-wrt-SBM}
There exists a set of independent standard Brownian motions,\\
$W^{(a)}_i(\cdot)$  and 
$W^{(s)}_{ij}(\cdot)$, constructed on the same probability space 
as the set of Poisson processes $\Pi^{(a)}_i(\cdot)$ and 
$\Pi^{(s)}_{ij}(\cdot)$, such that the following holds.
For any fixed $T>0$, as  $r\to\infty$:\\
 for each $i$
\beql{eq-strong-i}
\sup_{0 \leq t \leq T\log r} r^{-1/4} \abs{\Pi^{(a)}_i(r t) - rt
-W^{(a)}_i(r t)} 
\to 0, ~~~\mbox{w.p.1},
\end{equation}
and for each $(ij)\in \CE$
\beql{eq-strong-ij}
\sup_{0 \leq t \leq T\log r} r^{-1/4} \abs{\Pi^{(s)}_{ij}(r t) - rt
-W^{(s)}_{ij}(r t)} 
\to 0, ~~~\mbox{w.p.1}.
\end{equation}
\end{prop}

\begin{proof} This follows from Proposition~\ref{thm:strong approximation-111-clean}:
in its statement we 
replace $t$ with $rt$, $T$ with $rT\log r$, and $u$ with $r^{1/8}$. 
\end{proof}

\begin{prop}
\label{prop-SBM-on-logr}
Consider any sequence of standard Brownian motions,\\ $B_1(\cdot),B_2(\cdot),\ldots$,
defined on a common probability space. (They may be dependent.)
Let $T>0$, $C_{12}>0$ and $\epsilon\in (0,1/4)$ be fixed. 
Then, w.p.1-l.r,
conditions $t_1,t_2 \in [0,T\log r]$ and $|t_2-t_1|\le C_{12} r^{-1/2+\epsilon}$
imply that $|B_r(t_2)-B_r(t_1)|<r^{-1/8}$.
\end{prop}

\begin{proof} 
This follows from basic properties of Brownian motion.\\
Fix $\epsilon'\in (1/8,1/4-\epsilon/2)$.
Then for some fixed $C_{13}>0$,
\beql{eq-event}
\BP\{\max_{t\in[0,C_{12} r^{-1/2+\epsilon}]}|B_r(t)-B_r(0)| \ge r^{-\epsilon'}\}
\le \exp\{-C_{13} [r^{-\epsilon'}/r^{-1/4+\epsilon/2}]^2\}.
\end{equation}
This probability decays very fast with $r$. 
We divide the interval $[0,T\log r]$ into (polynomial in $r$ number of)
$C_{12} r^{-1/2+\epsilon}$-long subintervals, and use the above 
probability estimate for each of them; by Borel-Cantelli lemma,
w.p.1-l.r, the event (analogous to the event) in \eqn{eq-event} will not hold for any of the 
subintervals. The result follows. 
\end{proof}

\begin{proof}[Proof of Lemma~\ref{lem-close-to-diffusion}]
Suppose for each $r$ the initial state is fixed that
satisfies conditions of the lemma.
Suppose the process, for any $r$, is driven by a common set of 
Poisson processes, and associated Brownian motions
constructed on the same probability space, 
as specified in Proposition~\ref{prop-strong-wrt-SBM}.
It will suffice to show that for any subsequence of $r$, there exists a further
subsequence, along which the lemma conclusion holds.
So, let us fix an arbitrary subsequence of $r$.
We fix any $\nu_9\in (0,1/4)$
and choose a further subsequence of $r$, 
with $r$ increasing sufficiently fast,
so that w.p.1-l.r the events in the displayed formulas in Lemma~\ref{lem-sup-tight-very}
hold.

Denote:
$$
\hat A^r_i(t)=r^{-1/2}[\Pi^{(a)}_i(\lambda_i r t) - \lambda_i rt], 
~~~ \hat W^{(a),r}_i(t)=r^{-1/2} W^{(a)}_i(\lambda_i r t),
$$
$$
\hat D^r_{ij}(t)=r^{-1/2}[\Pi^{(s)}_{ij}(\mu_{ij} \psi_{ij}^* r t) - \mu_{ij} \psi_{ij}^* rt],
~~~ \hat W^{(s),r}_{ij}(t)=r^{-1/2} W^{(s)}_{ij}(\mu_{ij} \psi_{ij}^* r t).
$$
Note that, {\em for any $r$}, 
the {\em law} of 
$\left((\hat W^{(a),r}_i(\cdot)), (\hat W^{(s),r}_{ij}(\cdot))\right)$ 
is equal to that of\\
$\left((\sqrt{\lambda_i} B^{(a)}_i(\cdot)), (\sqrt{\mu_{ij} \psi_{ij}^*} B^{(s)}_{ij}(\cdot))\right)$,
where all $B^{(a)}_i(\cdot)$ and $B^{(s)}_{ij}(\cdot)$ are independent 
standard Brownian motions.

Using standard sample path representation (see e.g. \cite{PTW}), 
we can write, for each $i$, and all $t\ge 0$:
\beql{eq-integral-prelim-for-x-unscaled}
X_{i}^r(t) = X_{i}^r(0) + A^r_i(t)
- \sum_j D^r_{ij}\left(\mu_{ij} \int_0^t \Psi_{ij}^r(s) ds \right).
\end{equation}
Switching, again in a standard way, to diffusion-scaled variables
and to a ($I$-dimensional) vector form, we rewrite \eqn{eq-integral-prelim-for-x-unscaled}
as 
\beql{eq-integral-prelim-for-x}
(\hat X_{i}^r(t)) = (\hat X_{i}^r(0)) + (\hat A^r_i(t))
- \left( \sum_j \hat D^r_{ij}\left((\psi^*_{ij} r t)^{-1} [\int_0^t \Psi_{ij}^r(s) ds] t\right)
\right) - 
\end{equation}
$$
\left(\sum_j \int_0^t \mu_{ij} \hat \Psi_{ij}^r(s) ds\right).
$$

Suppose $\epsilon\in (0,1/4)$ (so that we can apply 
Proposition~\ref{prop-SBM-on-logr} later). We will make the choice of $\epsilon$
more specific below.

We claim that
w.p.1-l.r the following properties hold
uniformly for $t\in [0,T\log r]$: 
\beql{eq-error1}
|\hat A^r_i(t)-\hat W^{(a),r}_i(t)| < r^{-1/4}, ~\forall i,
~~~|\hat D^r_{ij}(t)-\hat W^{(s),r}_{ij}(t)| < r^{-1/4}, ~\forall (ij),
\end{equation}
\beql{eq-error2}
|(\psi^*_{ij} r t)^{-1} [\int_0^t \Psi_{ij}^r(s) ds]t - t| \le r^{-1/2+\epsilon} \epsilon T \log r
< r^{-1/2+\epsilon'}, ~\forall (ij),
\end{equation}
\beql{eq-error3}
|L'(\hat X_{i}^r(t)) - (\hat \Psi_{ij}^r(t))| < r^{-1/4},
\end{equation}
where $\epsilon'$ is a fixed number within $(\epsilon,1/4)$ and linear maping
$L'$ is defined by \eqn{eq-Lprime-def}. 
($L'$ was defined in \cite[Section 5.2]{SY2012}.
It maps a vector of {\em centered}
customer quantities into the vector of {\em centered} occupancies,
{\em assuming all queues and idlenesses in pools $j<J$ are zero.})
Indeed: properties \eqn{eq-error1} follow from Proposition~\ref{prop-strong-wrt-SBM};
property \eqn{eq-error2} follows from \eqn{eq-44444} in Lemma~\ref{lem-sup-tight-very};
property \eqn{eq-error3} follows from \eqn{eq-55555} in Lemma~\ref{lem-sup-tight-very}
and the definition of operator $L'$.

Using properties \eqn{eq-error1}-\eqn{eq-error3}, 
the sample path relation \eqn{eq-integral-prelim-for-x} implies
the following relation (written in vector form, with components
indexed by $(ij)$), which holds
w.p.1-l.r
uniformly for $t\in [0,T\log r]$:
\beql{eq-integral-prelim-psi}
(\hat \Psi_{ij}^r(t)) = (\hat \Psi_{ij}^r(0)) 
+ L'\left(\hat W^{(a),r}_i(t)\right) - L'\left(\sum_j \hat W^{(s),r}_{ij}(t)\right) -
\end{equation}
$$
 L'\left(\sum_j \int_0^t \mu_{ij} \hat \Psi_{ij}^r(s) ds \right)
+ (\Delta_i^r(t)),
$$
where $|(\Delta_i^r(t))|< r^{-1/9}$. (Instead of $1/9$ we could use any fixed number
in $(0,1/8)$.) Indeed, in \eqn{eq-integral-prelim-for-x} we can replace
$\hat A_i^r$ and $\hat D_{ij}^r$ with $\hat W_i^{(a),r}$ and $\hat W_{ij}^{(s),r}$,
respectively, which introduces an $o(r^{1/4})$ error by \eqn{eq-error1};
then, we apply operator $L'$ to both sides and replace $L'(\hat X_i^r)$ with
$(\hat \Psi_{ij}^r)$, which introduces an $o(r^{1/4})$ error by \eqn{eq-error3};
finally, we replace time $(\psi^*_{ij} r t)^{-1} [\int_0^t \Psi_{ij}^r(s) ds]t$
with $t$ in the argument of $\hat W_{ij}^{(s),r}$, which introduces an $O(r^{1/8})$ error by \eqn{eq-error2}
and Proposition~\ref{prop-SBM-on-logr}.

For each $r$ and each initial
condition $(\hat \Psi_{ij}^r(0))$, 
in addition to \eqn{eq-integral-prelim-psi}
consider the (strongly) unique strong solution
(by Theorems 5.2.9 and 5.2.5 of \cite{Karatzas_Shreve}) 
$(\breve \Psi_{ij}^r(\cdot))$ of the stochastic integral equation
\beql{eq-diffusion-psi}
(\breve \Psi_{ij}^r(t)) = (\breve \Psi_{ij}^r(0)) 
+ L'\left(\hat W^{(a),r}_i(t)\right) - L'\left(\sum_j \hat W^{(s),r}_{ij}(t)\right) -
\end{equation}
$$
 L'\left(\sum_j \int_0^t \mu_{ij} \breve \Psi_{ij}^r(s) ds \right),
$$
driven by the same set of Brownian motions 
$\left(\hat W^{(a),r}_i(\cdot),\hat W^{(s),r}_{ij}(\cdot)\right)$
and with the same initial condition $(\breve \Psi_{ij}^r(0))=(\hat \Psi_{ij}^r(0))$.
Thus, solutions to both \eqn{eq-integral-prelim-psi} and \eqn{eq-diffusion-psi},
for all $r$, 
are constructed on the same probability space associated with the underlying set 
of independent Brownian motions (and the corresponding Poisson processes coupled with them).
W.p.1-l.r we have for $t\in [0,T\log r]$:
$$
|(\hat \Psi_{ij}^r(t))-(\breve \Psi_{ij}^r(t))| \le 
|(\Delta_i^r(t))| + \int_0^t C' |(\hat \Psi_{ij}^r(s))-(\breve \Psi_{ij}^r(s))|ds,
$$
with some constant $C'>0$.
By Gronwall inequality (see e.g. Theorem 5.1 in Appendix 5 of \cite{Ethier_Kurtz}), for $t\in [0,\epsilon T\log r]$:
\beql{eq-gronwall-2}
|(\hat \Psi_{ij}^r(t))-(\breve \Psi_{ij}^r(t))| \le r^{-1/9} e^{C' \epsilon T \log r}
= r^{-1/9+\epsilon C' T}.
\end{equation}
We now specify the choice of $\epsilon$:  it is such that both $-1/8+\epsilon C' T<0$
and (for the reasons explained earlier) $\epsilon <1/4$ hold. In other words,
$0<\epsilon < \min\{1/4,1/(9C' T)\}$.
\end{proof}

Recall that {\em for any $r$} the {\em law} of the multi-dimensional Brownian motion\\
$\left(\hat W^{(a),r}_i(\cdot), \hat W^{(s),r}_{ij}(\cdot)\right)$,
driving equation \eqn{eq-diffusion-psi}, is same as that of\\
$\left(\sqrt{\lambda_i} B^{(a)}_i(\cdot), \sqrt{\mu_{ij} \psi_{ij}^*} B^{(s)}_{ij}(\cdot)\right)$,
where all $B^{(a)}_i(\cdot)$ and $B^{(s)}_{ij}(\cdot)$ are independent
standard Brownian motions. Therefore, {\em for any $r$}, the {\em law}
of the solution to \eqn{eq-diffusion-psi} is equal to that of the solution
to stochastic differential equation
\beql{eq-diffusion-psi-generic}
d (\breve \Psi_{ij}(t)) =
L' d \left(\sqrt{\lambda_i} B^{(a)}_i(t)\right) 
- L' d\left(\sum_j \sqrt{\mu_{ij} \psi_{ij}^*} B^{(s)}_{ij}(t)\right) -
\end{equation}
$$
 L' \left(\sum_j \mu_{ij} \breve \Psi_{ij}(t)\right) dt,
$$
with same initial state $(\breve \Psi_{ij}(0))=(\breve \Psi_{ij}^r(0))$.
This is equation \eqn{eq-diffusion-psi-generic-1}.
Moreover, the drift term in \eqn{eq-diffusion-psi-generic} can be written as
$$
- L'\left(\sum_j \mu_{ij} \breve \Psi_{ij}(t)\right)dt = L \left(\breve \Psi_{ij}(t)\right) dt,
$$
where matrix $L$ is easily checked to be 
exactly the 
matrix in the ODE $d (\tilde \psi_{ij}(t))= L (\tilde \psi_{ij}(t)) dt$
for the local fluid model, which follows from conditions (24) in \cite{SY2012}.
From \cite[Theorem 23]{SY2012} we know that 
{\em all eigenvalues of $L$ have negative real parts.}

\begin{prop}
\label{prop-stable-diffision}
Uniformly on all fixed initial conditions $(\breve \Psi_{ij}(0))$
from any fixed bounded set, the corresponding solutions 
to the stochastic differential equation \eqn{eq-diffusion-psi-generic}
have the following properties. Uniformly on all $t\ge 0$, 
the random vector $(\breve \Psi_{ij}(t))$ 
is Gaussian, with bounded mean and covariance matrix.
Moreover, as $t\to\infty$, 
the mean vector and the covariance matrix of $(\breve \Psi_{ij}(t))$
converge to those of the unique stationary distribution,
$Inv[(\breve \Psi_{ij}(\cdot))]$,
which is Gaussian with zero mean.
\end{prop}

\begin{proof}
This follows from the fact that all eigenvalues of the
drift matrix $L$ have negative real parts: see (5.6.12),
(5.6.13)', (5.6.14)', Problem 5.6.6 and Theorem 5.6.7 in \cite{Karatzas_Shreve}.
\end{proof}

\begin{proof}[Conclusion of the proof of Theorem~\ref{th-stationary-scale-main-strong}]
Consider Markov process $F^r(\cdot)$ in stationary regime.
We choose $T$ as in Lemma~\ref{lem-hit-sqrt-r}, 
then $\epsilon$ as in Lemma~\ref{lem-close-to-diffusion}, and consider the process
 in the interval $[0,\epsilon T \log r]$. Fix arbitrary $\nu_9\in (0,1/4)$.
The combination of \cite[Theorem 10(ii)]{SY2012}, Lemma~\ref{lem-hit-sqrt-r}
and Lemma~\ref{lem-down-to-subdiffusion-scale} shows
the following fact: uniformly on all sufficiently large $r$, 
the process will ``hit'' a state, satisfying conditions 
of Lemma~\ref{lem-close-to-diffusion}, with probability that 
can be made arbitrarily close to $1$ by choosing sufficiently large
fixed $C_{11}>0$.

Now, suppose at some time point
within $[0,\epsilon T \log r]$ the process is in a state
satisfying conditions 
of Lemma~\ref{lem-close-to-diffusion}. 
First, we obtain a bound on $|F^r(\epsilon T \log r)|$. Namely,
uniformly on all sufficiently large $r$, 
$|F^r(\epsilon T \log r)|\le C_{14} r^{1/2}$
with probability that
can be made arbitrarily close to $1$ by choosing sufficiently large
fixed $C_{14}>0$.
This follows
from Lemma~\ref{lem-close-to-diffusion} and
Proposition~\ref{prop-stable-diffision}.
This establishes the tightness of the sequence of 
$Inv[(\hat \Psi^r_{ij}(\cdot))]\equiv Inv[r^{-1/2} (\Psi^r_{ij}(\cdot)-\psi_{ij}^* r)]$.
Second, we obtain a bound on\\
$|(Q_i^r(\epsilon T \log r))|+ |(Z_j^r(\epsilon T \log r), j<J)|$.
This is even easier -- by \eqn{eq-55555} in Lemma~\ref{lem-sup-tight-very}
$$
\BP\{|(Q_i^r(\epsilon T \log r))|+ |(Z_j^r(\epsilon T \log r), j<J)|
\le C_{10}\delta_9 r^{\nu_9}\} \to 1.
$$
But, since $\nu_9$ can be chosen arbitrarily
small, we obtain property \eqn{eq-main-res2}.

Given the tightness of the sequence
of $Inv[(\hat \Psi^r_{ij}(\cdot))]$ and property \eqn{eq-main-res2}, it is straightforward
to show the remaining property \eqn{eq-main-res1}. 
(The argument is essentially same as that in the proof of 
 \cite[Theorem 8.5.1]{Liptser_Shiryaev}, although that result
does not directly apply to our setting.) Consider Markov process $F^r(\cdot)$ 
in stationary regime. We fix arbitrary $T>0$, $\delta_9>0$ and $\nu_9\in (0,1/4)$,
and then a large enough parameter 
$C_{11}>0$, so that, with probability arbitrarily close to $1$,
 the conditions on $F^r(0)$ in Lemma~\ref{lem-close-to-diffusion}
are satisfied for all large $r$. We then pick a sufficiently small fixed
$\epsilon>0$, so that property \eqn{eq-close-to-diffusion} holds.
Finally, using Proposition~\ref{prop-stable-diffision}, 
we pick a sufficiently large $T'>0$, 
so that $Dist[(\breve \Psi_{ij}(T'))]$ is close 
to $Inv[(\breve \Psi_{ij}(\cdot))]$,
uniformly on the initial states
$|(\breve \Psi_{ij}(0))| \le C_{11}$.
(Here 'close' is in the sense of close Gaussian distribution parameters, 
means and covariances; or, more generally, it can be in the sense of
Prohorov metric \cite{Ethier_Kurtz}.)
Note that, for all large $r$,
$T' < \epsilon T \log r$. Applying Lemma~\ref{lem-close-to-diffusion},
we see that, for all large $r$,
 $Dist[(\hat \Psi_{ij}^r(T'))]$ is close to
$Dist[(\breve \Psi_{ij}^r(T'))]$, which in turn is close to
$Inv[(\breve \Psi_{ij}^r(\cdot))]=Inv[(\breve \Psi_{ij}(\cdot))]$; 
and we can make it arbitrarily close by rechoosing
parameters. This implies \eqn{eq-main-res1}. We omit further details.
\end{proof}

\section{Discussion}
\label{sec-discussion}

As already mentioned in the Introduction, we believe that the approach 
developed in \cite{SY2012} and this paper provides quite generic scheme
for establishing diffusion-scale tightness of invariant distributions,
under the strictly subcritical load $\rho<1$. 
The approach shows that for the diffusion-scale tightness to hold,
it is essentially sufficient to verify the two key stability properties -- global stability
and local stability -- which we (at a high level and informally) describe next.
Let $F^r(\cdot)$ be a process describing the system state deviation from the 
equilibrium point. (For the LAP policy,
$F^r(t)=(\Psi^r_{ij}(t)-\psi_{ij}^* r, Q^r_i(t))$ as defined in this paper.)

{\em (a) Global stability.} The fluid limit $f(t), ~t\ge 0,$ is defined 
as\\ $\lim_r r^{-1} F^r(t), ~t\ge 0$. By global stability we mean the following property:
(a.1) the trajectories $f(t)$ converge to $0$, uniformly in the initial states
from a bounded set. Moreover, we also require the following
related property to hold: (a.2) uniformly on all infinite initial states, $|f(0)|=\infty$,
each trajectory $f(t)$ reaches a state,
where all server pools are fully occupied, and then stays in such a state forever.
(For the LAP policy, the formal statements are 
\cite[Propositions 13 and 16]{SY2012}.)

{\em (b) Local stability.} Suppose $h(r)$ is a function of $r$ such that
$h(r)/r\to 0$ and $h(r)/\sqrt{r} \to \infty$. The local fluid limit
$\tilde f(t), ~t\ge 0,$ is defined as\\ $\lim_r h(r)^{-1} F^r(t), ~t\ge 0$. 
Suppose, the trajectories $\tilde f(\cdot)$ satisfy a linear ODE 
$(d/dt)\tilde f(t) = L \tilde f(t)$. By local stability we 
mean the property that all eigenvalues of $L$ have negative real parts.
(For the LAP, the formal statement is \cite[Theorem 23]{SY2012}.
For the LQFS-LB policy of \cite{SY10}, the local stability does {\em not} hold.)

Properties (a) and (b) may or may not be easy to verify
for a given control policy; but the task of proving or disproving them
is typically much easier than the full task of verifying the diffusion-scale
tightness. We also note that showing local stability may require working
with the process under additional space and/or time scalings, 
such as hydrodynamic scaling for LAP (see \cite[Section 5.2]{SY2012}).

If the global and local stability properties hold,
the steps of establishing diffusion-scale tightness of invariant distributions
are as follows.

{\em Step 1. Existence and $o(r)$-scale tightness of invariant distributions.}
Using the global stability property (a.2)
and employing the total (appropriately defined) workload
in the system as a Lyapunov function, one can prove the positive recurrence
(stochastic stability) of the process, and therefore existence of a stationary 
distribution. The proof is fairly standard, uses
 Lyapunov function average drift argument, which additionally shows that
$\BE |r^{-1} F^r|$ is bounded, which in turn applies the tightness
of distributions of $r^{-1} F^r$. We then employ the global stability property
(a.1) to show that, in fact, the invariant distributions of 
$r^{-1} F^r$ asymptotically concentrate at $0$. 
This can be referred to as $o(r)$-scale tightness.
(The formal result for LAP is in \cite[Theorem 14]{SY2012}.)

{\em Step 2. $r^{1/2+\epsilon}$-scale tightness.} Local stability implies 
exponentially fast convergence of fluid limit trajectories  $\tilde f(\cdot)$
to $0$. In particular, 
for a sufficiently large fixed $T$, the norm $|\tilde f(t+T)|\le \delta |\tilde f(t)|$, 
where $\delta<1$.
We use this, and probability estimates for deviations of $h(r)^{-1} F^r(t)$
from a corresponding local fluid limit $\tilde f(t)$, to show that 
if $F^r(0)=h(r)=o(r)$ then with high probability $|F^r(T)| \le \delta |F^r(0)|$.
Now, it takes $O(\log r)$ intervals of length $T$ for $|F^r|$ to ``descend''
from $o(r)$ to $r^{1/2+\epsilon}$, and we show that this does in fact happen with 
high probability. (So, the key technical issue here is that we 
have to do probability estimates not on a finite, but on an $O(\log r)$ interval.)
This implies $r^{1/2+\epsilon}$-scale tightness, for any $\epsilon>0$;
namely, the invariant distributions of 
$r^{-1/2-\epsilon} F^r$ asymptotically concentrate at $0$.
(The formal argument for LAP is in \cite[Section 5.2]{SY2012}.)
Note that this property is {\em weaker} than, for example,
$\BE |r^{-1/2-\epsilon} F^r| \to 0$.

{\em Step 3. Diffusion-scale ($r^{1/2}$-scale) tightness.} 
Here we start with the $r^{1/2+\epsilon}$-scale tightness, with $\epsilon>0$
being sufficiently small. We show that if $|F^r(0)|=O(r^{1/2+\epsilon})$, then,
with high probability, $|F^r(t)|$ ``hits the diffusion scale'' 
$O(r^{1/2})$ within $\epsilon \log r$.
Again, this is achieved by considering $O(\log r)$ consecutive $T$-long
intervals, in each of which $|F^r|$ must decrease by a factor with high probability,
unless $|F^r(t)|$ does hit $O(r^{1/2})$. (The formal result for LAP
is Lemma~\ref{lem-hit-sqrt-r}.) Given that, it remains to show that 
if $|F^r(0)|=O(r^{1/2})$ and $\epsilon$ is small enough, 
then for any $t\in [0,\epsilon \log r]$, we also have
$|F^r(t)|=O(r^{1/2})$ with high probability. This is done by showing
the closeness of process $r^{-1/2} F^r(\cdot)$ to the corresponding limiting
diffusion process {\em on the $\epsilon \log r$-long interval}, 
and the fact that the drift matrix of the diffusion
process is exactly the $L$ matrix from the definition of local stability. 
(For LAP, this takes the bulk of this paper,
from Lemma~\ref{lem-down-to-subdiffusion-scale} on.
It involves, in particular, showing that all queues and all pool idlenesses,
except for pool $J$ serving the lowest priority activity, are in fact $o(r^{\nu})$ 
for any $\nu>0$.)
Again, we note that the diffusion-scale tightness is {\em weaker} than, for example,
the boundedness of $\BE |r^{-1/2} F^r|$.

In conclusion, we remark again that many (although not all) parts of the above scheme
do rely on the strict subcriticality condition $\rho<1$.
It would be of interest to explore whether the approach can be extended to
establishing diffusion-scale tightness in the Halfin-Whitt regime.

\bibliographystyle{apt}
\bibliography{tightness-strong}

\begin{thebibliography}{10}

\bibitem{AksinArmonyMehrotra}
{\sc Aksin, Z., Armony, M. and Mehrotra, V.} (2007).
\newblock The modern call-center: A multi-disciplinary perspective on
  operations management research.
\newblock {\em Production and Operations Management, Special Issue on Service
  Operations in honor of John Buzacott (ed. G. Shanthikumar and D. Yao)\/} {\bf
  16,} 655--688.

\bibitem{Armony_Ward}
{\sc Armony, M. and Ward, A.}
\newblock Blind fair routing in large-scale service systems.
\newblock Preprint October 2011.
\newblock \url{http://www-bcf.usc.edu/~amyward/ArWa_10_6_11}.

\bibitem{Atar2009}
{\sc Atar, R., Shaki, Y. and Shwartz, A.} (2011).
\newblock A blind policy for equalizing cumulative idleness.
\newblock {\em Queueing Systems\/} {\bf 67,} 275--293.

\bibitem{Csorgo_Horvath}
{\sc Cs{\"{o}}rg{\H{o}}, M. and Horv{\'{a}}th, L.} (1993).
\newblock {\em Weighted approximations in probability and statistics}.
\newblock Wiley.

\bibitem{Ethier_Kurtz}
{\sc Ethier, S. and Kurtz, T.} (1986).
\newblock {\em Markov Processes. Characterization and Convergence}.
\newblock Wiley.

\bibitem{GamarnikGoldbergGGN}
{\sc Gamarnik, D. and Goldberg, D.} (2013).
\newblock Steady-state {GI/GI/n} queue in the {Halfin-Whitt} regime.
\newblock {\em Annals of Applied Probability\/}.

\bibitem{Gamarnik_Momcilovic}
{\sc Gamarnik, D. and Momcilovic, P.} (2008).
\newblock Steady-state analysis of a multiserver queue in the halfin-whitt
  regime.
\newblock {\em Advances in Applied Probability\/} {\bf 40,} 548--577.

\bibitem{Gamarnik_Stolyar}
{\sc Gamarnik, D. and Stolyar, A.~L.} (2012).
\newblock Multiclass multiserver queueing system in the {Halfin-Whitt} heavy
  traffic regime. asymptotics of the stationary distribution.
\newblock {\em Queueing Systems\/} {\bf 71,} 25--51.

\bibitem{GansKooleMandelbaum}
{\sc Gans, N., Koole, G. and Mandelbaum, A.} (2003).
\newblock Telephone call centers: Tutorial, review, and research prospects.
\newblock {\em Manufacturing \& Service Operations Management\/} {\bf 5,}
  79--141.

\bibitem{Gurvich_Whitt}
{\sc Gurvich, I. and Whitt, W.} (May 2009).
\newblock Queue-and-idleness-ratio controls in many-server service systems.
\newblock {\em Mathematics of OR\/} {\bf 34,} 363--396.

\bibitem{Karatzas_Shreve}
{\sc Karatzas, I. and Shreve, S.} (1996).
\newblock {\em Brownian Motion and Stochastic Calculus (2nd ed.)}.
\newblock Springer.

\bibitem{Liptser_Shiryaev}
{\sc Liptser, R.~S. and Shiryaev, A.~N.} (1989).
\newblock {\em Theory of Martingales}.
\newblock Kluwer Academic Publishers.

\bibitem{PTW}
{\sc Pang, G., Talreja, R. and Whitt, W.} (2007).
\newblock Martingale proofs of many-server heavy-traffic limits for markovian
  queues.
\newblock {\em Probability Surveys\/} {\bf 4,} 193--267.

\bibitem{ShSt2000}
{\sc Shakkottai, S. and Stolyar, A.~L.} (2002).
\newblock Scheduling for multiple flows sharing a time-varying channel: the
  exponential rule.
\newblock {\em Analytic Methods in Applied Probability. In Memory of Fridrih
  Karpelevich. Yu. M. Suhov, Editor. American Mathematical Society
  Translations, Series 2\/} {\bf 207,} 185--202.

\bibitem{Stolyar_Tezcan_underload}
{\sc Stolyar, A.~L. and Tezcan, T.} (2010).
\newblock Control of systems with flexible multi-server pools: a shadow routing
  approach.
\newblock {\em Queueing Systems\/} {\bf 66,} 1--51.

\bibitem{Stolyar_Tezcan}
{\sc Stolyar, A.~L. and Tezcan, T.} (2011).
\newblock Shadow routing based control of flexible multi-server pools in
  overload.
\newblock {\em Operations Research\/} {\bf 59,} 1427--1444.

\bibitem{SY2012}
{\sc Stolyar, A.~L. and Yudovina, E.} (2012).
\newblock Tightness of invariant distributions of a large-scale flexible
  service system under a priority discipline.
\newblock {\em Stochastic Systems\/} {\bf 2,} 381--408.

\bibitem{SY10}
{\sc Stolyar, A.~L. and Yudovina, E.} (2013).
\newblock Systems with large flexible server pools: Instability of
  {``}natural{''} load balancing.
\newblock {\em Annals of Applied Probability\/} {\bf 23,} 2099--2138.

\end{thebibliography}

\end{document}